\documentclass{amsart}
\usepackage{tikz,bm,amsmath,amsrefs}

\newtheorem{theorem}{Theorem}
\newtheorem{lemma}{Lemma}
\newtheorem{example}{Example}
\newtheorem{definition}{Definition}
\newtheorem{cor}{Corollary}
\newtheorem{rem}{Remark}
\newtheorem{proposition}{Proposition}
\newtheorem*{claim}{Claim}

\def\ld{\mathbin{\backslash}}
\def\rd{\mathbin{/}}

\newcommand{\bl}{\underline{\hspace{1ex}}}

\newcommand{\blambda}{\boldsymbol{\lambda}}
\newcommand{\brho}{\boldsymbol{\rho}}

\title{Two-sided wreath product done right}
\author{Michal Botur, Tomasz Kowalski}

\begin{document}

\begin{abstract}
We investigate a semigroup construction related to the two-sided
wreath product. It encompasses a range of known constructions and gives a 
slightly finer version of the decomposition in the Krohn-Rhodes Theorem,
in which the three-element
flip-flop is replaced by the two-element semilattice. We develop foundations of
the theory of our construction, showing in the process that
it naturally combines ideas from semigroup theory
(wreath products), category theory (Grothendieck construction), and
ordered structures (residuated lattices). 
\end{abstract}  

\maketitle

\section{Introduction}\label{intro}

The purpose of this article is to introduce and investigate a certain semigroup
construction which encompasses a range of known constructions including
transformation monoids, semigroup actions, and wreath products. We chose the
cheeky (but also tongue-in-cheek) title because our construction
is inspired by the standard way of presenting the wreath product, 
say, of groups, as a direct power $G^X$ together with a group $K$ acting on $X$,
that is, a set of bijective maps $X\to X$, indexed by
elements of $K$. For semigroups, the restriction to bijections seems artificial: after
all, semigroups are representable as semigroups of arbitrary maps. And if the
maps do not have to be surjective, there seems to be no reason for having the
same set of coordinates for every element of $K$.  

A rudimentary construction of this type has been used in~\cite{JM06}
to settle some questions about \emph{generalised BL-algebras}, which are
a subclass of certain special lattice-ordered monoids known as
\emph{residuated lattices}. For the purposes of this article, familiarity with
residuated lattices is not necessary, but the interested reader is referred
to~\cite{JT02} for a very readable albeit slightly old survey.

The construction was expanded and
investigated in~\cite{DK14}, under the name of \emph{kites}, still in the
context of residuated lattices. A series of applications and further
generalisations followed, see, e.g.,~\cite{DH14} and~\cite{BD15}.
A modification of the kite construction (to be precise, a subsemigroup of a
kite) was put to a good use in~\cite{BKLT16}. All these, however,
stayed within the area of ordered structures, and the interaction of
multiplication with order was the main focus. It was clear from the beginning
that the kite construction is closely related to wreath products of 
ordered structures, for example, from~\cite{JT04}, or for a more specific case of
lattice-ordered groups, from~\cite{HMC69}.
Considering order, however, seems
to have obscured the properties of the multiplicative structure to some extent.

Here we depart from order (in the content, not in the organisation)
and investigate only the multiplicative structure. As an application,
we will show that the decomposition of finite semigroups in the celebrated
Krohn-Rhodes Theorem (originally in~\cite{KR62}, see also~\cite{RS09}) can be
given a slightly finer form, namely, the flip-flop monoids can be replaced by
two-element semilattices.

\subsection{Notation}
We use the category-theoretic notation for composition of maps, that is, 
for maps $f\colon A\longrightarrow B$ and $g\colon B\longrightarrow C$
we denote their composition by $g\circ f\colon A\longrightarrow C$, so that
$(g\circ f)(a) = g(f(a))$ for all $a\in A$. The set of all
maps from the $A$ to $B$ we denote by the usual $B^A$. For a map
$f\colon A\longrightarrow B$ and a set $I$ we write
$f^I\colon A^I\to B^I$ for the map defined by $f^I(x)(i)=f(x(i))$.
The following easy proposition will be used repeatedly without further ado.

\begin{proposition}\label{P1}
Let $\mathbf{G} = (G;\cdot)$ be a groupoid, and let $I$, $J$ be sets. 
Then for all $x,y\in G^I$ and any $f\in I^J$ the following equality holds
$$(x\circ f)\cdot (y\circ f)= (x\cdot y)\circ f.$$
\end{proposition}

We will frequently use systems of parameterised maps. In order to
distinguish easily between parameters and arguments, we will put the parameters
in square brackets, so $f[a,b](x)$ will denote the value of a
map $f[a,b]$ on the argument $x$.

We will also frequently pass between algebras (semigroups), categories, and
other types of structures (systems of maps), typically related to one
another. To help distinguishing between them, we will use different fonts.
Typically, boldface will be used for algebras (and italics for their
universes), sans serif will be used for categories, and script for other types
of structures. A few exceptions to these rules will be natural enough not to
cause confusion.

\subsection{The  main construction}

Let $\mathbf{S} = (S,\cdot)$ be a semigroup, and let 
$(I[s])_{s\in S}$ be an indexed system of sets. For any $(a,b)\in S\times S$, let 
$\lambda[a,b]\colon I[ab]\to I[a]$ and $\rho[a,b]\colon I[ab]\to I[b]$
be maps satisfying the following conditions
\begin{enumerate}
\item[($\alpha$)] $\lambda[a,b]\circ\lambda[ab,c] = \lambda[a,bc]$
\item[($\beta$)] $\rho[b,c]\circ\rho[a,bc] = \rho[ab,c]$
\item[($\gamma$)] $\rho[a,b]\circ\lambda[ab,c] = \lambda[b,c]\circ\rho[a,bc]$
\end{enumerate}
which make the diagram in Figure~\ref{l-r-system} commute. 
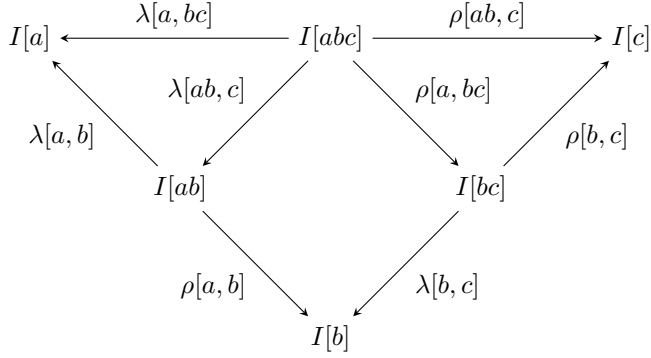
\begin{figure}
\begin{tikzpicture}[>=stealth,auto]
\node (I_abc) at (0,0) {$I[abc]$};
\node (I_a) at (-4,0) {$I[a]$};
\node (I_c) at (4,0) {$I[c]$};
\node (I_ab) at (-2,-2) {$I[ab]$};
\node (I_bc) at (2,-2) {$I[bc]$};
\node (I_b) at (0,-4) {$I[b]$};
\draw[->] (I_abc) to node[swap] {$\lambda[a,bc]$} (I_a);
\draw[->] (I_abc) to node {$\rho[ab,c]$} (I_c);
\draw[->] (I_abc) to node[swap] {$\lambda[ab,c]$} (I_ab);
\draw[->] (I_abc) to node {$\rho[a,bc]$} (I_bc);
\draw[->] (I_ab) to node {$\lambda[a,b]$} (I_a);
\draw[->] (I_ab) to node[swap] {$\rho[a,b]$} (I_b);
\draw[->] (I_bc) to node[swap] {$\rho[b,c]$} (I_c);
\draw[->] (I_bc) to node {$\lambda[b,c]$} (I_b);
\end{tikzpicture}
\caption{A $\lambda\rho$-system}
\label{l-r-system}
\end{figure}
Any triple $(\mathbf{I},\blambda,\brho)$ of systems of sets and maps
satisfying the above conditions will be called
a \emph{$\lambda\rho$-system over} $\mathbf{S}$. A \emph{(general)
  $\lambda\rho$-system} 
is then a quadruple $(\mathbf{S},\mathbf{I},\blambda,\brho)$,
where $\mathbf{S}$ is a semigroup and
$(\mathbf{I},\blambda,\brho)$ is a $\lambda\rho$-system over $\mathbf{S}$.
We will typically use script letters to refer to $\lambda\rho$-systems,
together with the convention that a $\lambda\rho$-system over a semigroup will
be referred to by the script variant of the letter naming the semigroup.
Thus, a $\lambda\rho$-system over $\mathbf{S}$ will be generally called
$\mathcal{S}$; subscripts, and occasionally other devices, will be used to
distinguish between different  
$\lambda\rho$-systems over the same semigroup.
Where convenient, we will also use a more explicit notation
$$
\bigl(\langle \lambda[a,b],\rho[a,b]\rangle\colon
I[ab]\longrightarrow I[a]\times I[b]\bigr)_{(a,b)\in S^2}
$$
for a $\lambda\rho$-system over a semigroup $\mathbf{S}$.

\begin{definition}\label{rl-prod}
Let $\mathbf{S} = (S;\cdot)$ be a semigroup and let
$\mathcal{S} = (\mathbf{I},\blambda,\brho)$
be a $\lambda\rho$-system over $\mathbf{S}$. Let  
$\mathbf{H}$ be a semigroup. Then, we define a groupoid 
$\mathbf{H}^{[\mathcal{S}]} = (H ^{[\mathcal{S}]};\star)$, by putting 
\begin{itemize}
\item $H^{[\mathcal{S}]} = \biguplus_{a\in S} H^{I[a]} = \{(x,a)\colon a\in S,\
  x\in H^{I[a]}\}$, and 
\item $(x,a)\star(y,b) = \bigl((x\circ\lambda[a,b])\cdot(y\circ\rho[a,b]),ab\bigr)$.
\end{itemize}
\end{definition}

We will call $\mathbf{H}^{[\mathcal{S}]}$ a \emph{$\lambda\rho$-product}. As the name
suggests, $\lambda\rho$-products are closely related to wreath products.
We will explore their relationship more closely in Section~\ref{wreath}.

For any $\lambda\rho$-system $\mathcal{S}$ over a semigroup $\mathbf{S}$, we
will call $\mathbf{S}$ the \emph{skeleton} of $\mathcal{S}$. We will extend this
terminology to $\lambda\rho$-products, that is, for any semigroup
$\mathbf{H}$, we will also call $\mathbf{S}$ the skeleton of
$\mathbf{H}^{[\mathcal{S}]}$.

\begin{theorem}\label{main}
Let $\mathbf{S} = (S;\cdot)$ be a semigroup and let
$$
\mathcal{S} = 
\bigl(\langle\lambda[a,b],\rho[a,b]\rangle\colon 
I[ab]\to I[a]\times I[b]\bigr)_{(a,b)\in S^2}
$$
be a system of sets and maps indexed by the elements of $S$. 
Then, the following are equivalent.
\begin{enumerate}
\item $\mathbf{H}^{[\mathcal{S}]}$ is a semigroup, for any semigroup $\mathbf{H}$.
\item $\bigl(\langle\lambda[a,b],\rho[a,b]\rangle\colon 
I[ab]\to I[a]\times I[b]\bigr)_{(a,b)\in S^2}$ is a $\lambda\rho$-system over $\mathbf{S}$.
\end{enumerate}
\end{theorem}

\begin{proof} First, note that associativity of the operation
  $\star$ is equivalent to the statement that the equality
\begin{equation}\label{assoc-eq}\tag{\dag}  
\begin{split}
&\bigl((x\circ\lambda[a,b]\circ\lambda[ab,c])
   \cdot(y\circ\rho[a,b]\circ\lambda[ab,c]) 
   \cdot(z\circ\rho[ab,c]),\ abc\bigr) = \\
&\bigl((x\circ\lambda[a,bc])
   \cdot(y\circ\lambda[b,c]\circ\rho[a,bc]) 
   \cdot(z\circ\rho[a,b]\circ\rho[a,bc]),\ abc\bigr)
\end{split}
\end{equation} 
holds for arbitrary $(x,a), (y,b), (z,c)\in H^{[\mathcal{S}]}$. 
To see it, we carry out the following straightforward calculation:
\begin{align*}
\bigl((x,a)\star(y,b)\bigr)&\star(z,c) = 
\bigl((x\circ\lambda[a,b])\cdot(y\circ\rho[a,b]),\ ab\bigr)\star (z,c) \\
&= \biggl(\Bigl(\bigl((x\circ\lambda[a,b])\cdot(y\circ\rho[a,b])\bigr)
   \circ\lambda[ab,c]\Bigr) 
   \cdot\bigl(z\circ\rho[ab,c]\bigr),\ abc\biggr) \\
&= \bigl((x\circ\lambda[a,b]\circ\lambda[ab,c])
   \cdot(y\circ\rho[a,b]\circ\lambda[ab,c]) 
   \cdot(z\circ\rho[ab,c]),\ abc\bigr) \\
&= \bigl((x\circ\lambda[a,bc])
   \cdot(y\circ\lambda[b,c]\circ\rho[a,bc]) 
   \cdot(z\circ\rho[b,c]\circ\rho[a,bc]),\ abc\bigr)\\
&= \biggl(\bigl(x\circ\lambda[a,bc]\bigr)
   \cdot\Bigl(\bigl((y\circ\lambda[b,c]) 
   \cdot(z\circ\rho[b,c])\bigr)\circ\rho[a,bc]\Bigr),\ abc\biggr) \\
&= (x,a)\star
   \bigl((y\circ\lambda[b,c])\cdot(z\circ\rho[b,c]),\ bc\bigr)\\
&= (x,a)\star\bigl((y,b)\star(z,c)\bigr)  
\end{align*}
where the only non-definitional equality is precisely~(\ref{assoc-eq}).
Now, if $\mathcal{S}$ is a  $\lambda\rho$-system, then~(\ref{assoc-eq}) 
follows immediately from the equations ($\alpha$), ($\beta$) and ($\gamma$).
This proves that (2) implies (1). The converse is clear.
\end{proof}

Note that, in general, neither
$\mathbf{S}$ nor $\mathbf{H}$ is a subsemigroup of 
$\mathbf{H}^{[\mathcal{S}]}$. However, it is not difficult to show that if
either of them is a monoid then the other one is a subsemigroup of 
$\mathbf{H}^{[\mathcal{S}]}$.

\begin{example}\label{skel}
Let $\mathbf{S}$ be a semigroup, and let $\mathbf{1}$ be the trivial semigroup.
Then, for any $\lambda\rho$-system $\mathcal{S}$ over $\mathbf{S}$ we have
$\mathbf{1}^{[\mathcal{S}]}\cong \mathbf{S}$. Indeed, $\mathbf{1}^I\cong
\mathbf{1}$ for any $I$, so
$\mathbf{1}^{[\mathcal{S}]} = \bigl(\{(1,s)\colon s\in S\},\star\bigr)$, with
$(1,a)\star(1,b) = (1,ab)$.  
\end{example}

The same effect can be achieved in a more fanciful way. 

\begin{example}\label{empty}
Let $\mathbf{S}$ be a semigroup, and let $I[s] = \emptyset$ for each $s\in S$.
Then, $\mathcal{S} = (\mathbf{I},\blambda,\brho)$, where $\lambda[a,b]$,
$\rho[a,b]$ are empty functions for each $(a,b)\in S^2$, is a
$\lambda\rho$-system over $\mathbf{S}$.  
For any semigroup $\mathbf{H}$ we then have that $H^{I[s]}$ is a singleton for each
$s\in S$ \textup{(}its only element is the empty map
$\emptyset\colon \emptyset\to H$\textup{)}. 
Moreover, $(\emptyset,a)\star(\emptyset,b) = (\emptyset, ab)$, for
any $a,b\in S$, and thus $\mathbf{H}^{[\mathcal{S}]}\cong \mathbf{S}$. 
\end{example}  

In either of these ways, every semigroup $\mathbf{S}$ is isomorphic to
a $\lambda\rho$-product whose skeleton is $\mathbf{S}$. One can ask how much
freedom there is for making some, but not necessarily all, sets $I[s]$ empty. The
answer is due to Dominik Lachman~\cite{Lachman}.

\begin{proposition}
Let $\mathcal{S} = (\mathbf{I},\blambda,\brho)$ be a $\lambda\rho$-system over a
semigroup $\mathbf{S}$. Let $J = \{s\in S\colon I[s] = \emptyset\}$.
If $J$ is nonempty, then $J$ is a two-sided ideal of\/ $\mathbf{S}$.
\end{proposition}

\begin{example}\label{prod}
Let $\mathbf{S}$ be a semigroup, and let $I[s] = \{1\}$ for each $s\in S$.
Then, $\mathcal{S} = (\mathbf{I},\blambda,\brho)$, where $\lambda[a,b]$,
$\rho[a,b]$ are constant functions for each $(a,b)\in S^2$, is a
$\lambda\rho$-system over $\mathbf{S}$.  
Then, $H^{I[s]}$ is a copy of $H$, for any semigroup $\mathbf{H}$. 
Moreover, for 
any $a,b\in S$ and $x,y\in H$, we have 
$(x,a)\star(y,b) = (xy, ab)$,
and thus $\mathbf{H}^{[\mathcal{S}]}\cong \mathbf{H}\times\mathbf{S}$. 
\end{example}

\begin{example}\label{lzero}
Let $\mathbf{1}$ be the trivial semigroup, and let 
$I = \{0,1\}$. Next, let $\lambda\colon I\to I$ be the identity map,
and let $\rho\colon I\to I$ be the constant map $\overline{0}$.
This defines a $\lambda\rho$-system $\mathcal{I}$ over $\mathbf{1}$.
Consider $\mathbb{Z}_2^{[\mathcal{I}]}$, whose universe
$\mathbb{Z}_2^I$ we will identify in the obvious way with
the set $\{00,01,10,11\}$. Here is the multiplication table of
$\mathbb{Z}_2^{[\mathcal{I}]}$:
$$
\begin{tabular}{c|cccc}
$\star$ & $00$  & $11$ & $01$ & $10$ \\
\hline                            
$00$  & $00$  & $11$ & $00$ & $11$ \\
$11$  & $11$  & $00$ & $11$ & $00$ \\  
$01$  & $01$  & $10$ & $01$ & $10$ \\
$10$  & $10$  & $01$ & $10$ & $01$ \\
\end{tabular}  
$$
Partitioning the universe into $\{00,11\}$ and $\{01,10\}$, we
obtain a congruence $\theta$, such that 
$\mathbb{Z}_2^{[\mathcal{I}]}/\theta$ is isomorphic to the two-element left-zero semigroup.
\end{example}

\begin{example}\label{flip-flop}
Let $\mathbf{2}=(\{0,1\},\vee)$ be the two-element join-semilattice, and let
$\mathcal{Z}$ be the $\lambda\rho$-system over $\mathbf{2}$, defined by putting
\begin{enumerate}
\item $I[0]=\{0\},$ $I[1]=\{0,1\}$,
\item $\lambda[1,0]=\rho[0,1]=\lambda[1,1]=id_{I[1]}$ and
  $\rho[1,1] = \overline{0}$. 
\end{enumerate}
This defines a unique $\lambda\rho$-system, since all
the remaining maps all have $\{0\}$ as the range. It is
easy to show that the semigroup $\mathbb Z_2^{[\mathcal{Z}]}$ is the following:
$$
\begin{array}{c|cccccc}
\star &0&1&00&11&01&10\\
\hline
0&0&1&00&11&01&10\\
1&1&0&11&00&10&01\\
00&00&11&00&11&00&11\\
11&11&00&11&00&11&00\\
01&01&10&01&10&01&10\\
10&10&01&10&01&10&01
\end{array}
$$
Partitioning the universe into $\{0,1\}$, $\{00,11\}$ and $\{01,10\}$
we obtain a congruence $\theta$, such that $\mathbb Z_2^{[\mathcal{Z}]}/\theta$ is
isomorphic to the left flip-flop monoid. 
\end{example}

In the commonly used terminology, Examples~\ref{lzero} and~\ref{flip-flop} show,
respectively, that the two-element left-zero semigroup strongly divides
a $\lambda\rho$-product of $\mathbb{Z}_2$ over the trivial semigroup, and
the three-element left flip-flop monoid strongly divides 
a $\lambda\rho$-product of $\mathbb{Z}_2$ over a two-element semilattice.
Thus, the flip-flop monoid turns out to be decomposable, in this sense.
It will be shown in Section~\ref{wreath} that wreath product
is also a special case of $\lambda\rho$-product.
The next two last examples pave the way to wreath products. 

\begin{example}\label{sgrp-act}
Let a semigroup $\mathbf{S}$ act on a set $X$ on the left. The system of maps
$$
\bigl(\langle \lambda[a,b],\rho[a,b]\rangle\colon
I[ab]\to I[a]\times I[b]\bigr),
$$
where $I[s]=X$ for any $s\in S$, and 
\begin{enumerate}
\item $\lambda[a,b] = {id}_X$ for any $a,b\in S$,
\item $\rho[a,b] = \bl\cdot a$ for all $a,b\in S$. 
\end{enumerate}
is a $\lambda\rho$-system over $\mathbf{S}$. An analogous
$\lambda\rho$-system is induced by $\mathbf{S}$ acting on the right.
\end{example}

Recall that a two-sided action of a semigroup
$\mathbf{S}$ on a set $X$ is 
a pair of maps $\ld\colon S\times X\longrightarrow X$ and
$\rd\colon X\times S\longrightarrow X$, satisfying 
$$
a\ld (b\ld x) = (b\cdot a)\ld x\qquad (x\rd a)\rd b = x\rd (b\cdot a)\qquad
(a\ld x)\rd b=a\ld (x\rd b)
$$
for any $a,b\in S$  and $x\in X$. The slash notation is not the commonest, but
we use it because of the connection with residuation, to come in
Example~\ref{rl-example}.

\begin{example}\label{sgrp-act-two-sided}
Let $(X,\ld,\rd,\mathbf{S})$ consist of a set $X$ together with a two-sided
action of a semigroup $\mathbf{S}$ on $X$. Then the system of maps
$$
\bigl(\langle \lambda[a,b],\rho[a,b]\rangle\colon
I[ab]\to I[a]\times I[b]\bigr),
$$
where $I[s]=X$ for any $s\in S$, and 
\begin{enumerate}
\item $\lambda[a,b] = \bl\rd b$ for any $a,b\in S$,
\item $\rho[a,b] = a\ld\bl$ for all $a,b\in S$. 
\end{enumerate}
is a $\lambda\rho$-system over $\mathbf{S}$.
\end{example}

The next example comes from the theory of ordered structures, more precisely,
residuated lattices. We present it mainly because it is the closest to the
kite construction that motivated the present work. The reader unfamiliar with
residuated lattices can safely skip the example.

\begin{example}\label{rl-example}
Let $(L;\cdot,\ld,\rd,\wedge,\vee,e)$ be a residuated lattice. For each $x\in L$
we put 
$I[x] = \{a\in L\colon x\leq a\}$. Next, let $\lambda[a,b] = \bl\rd b$ and
$\rho[a,b] = a\ld\bl$. Then $(\mathbf{I},\blambda,\brho)$ is a
$\lambda\rho$-system over the skeleton $(L;\cdot)$. The same holds for the more
general case of a partially ordered residuated semigroup
$(L;\leq, \cdot,\ld,\rd)$.
\end{example}

The last example in this section is hardly more than a curiosity, but we find it
quite illustrative. Let $\star$ be any semigroup operation on a
two-element Boolean algebra $\mathbf{B}$, say, meet, join, projection, or
addition modulo 2. Then, for any set $X$, on the one hand
$\star$ is a pointwise operation in $\mathbf{B}^X$, but on the
other hand, it has its \emph{alter ego} in the powerset $2^X$, via
characteristic functions. Here is an analogue of this for a
$\lambda\rho$-system over $\mathbf{B}$.
\begin{example}
Let $\mathcal{S} = (\mathbf{I},\blambda,\brho)$ be any
$\lambda\rho$-system, with a skeleton $\mathbf{S}$. Let $\star$ be any semigroup
operation of the two-element Boolean algebra $\mathbf{B}$. Then, 
$\mathbf{B}^{[\mathcal{S}]}$ is a semigroup whose universe is the disjoint union
of $2^{I[x]}$ for the system  $\{I[a]\colon a\in S\}$. The semigroup operation  
can be explicitly written as
$$
(U,a)\star (W,b) = \bigl(\lambda[a,b]^{-1}(U)\star\rho[a,b]^{-1}(W), ab\bigr) 
$$
where $U\subseteq I[a]$ and $W\subseteq I[b]$.
\end{example}  
One may think of the preimages $\lambda[a,b]^{-1}(U)$ and $\rho[a,b]^{-1}(W)$
as shadows cast by $U$ and $W$ in a stack of Venn diagrams.

\section{Categorical background}\label{cats}

In this short section we use some categorical tools to show that general
$\lambda\rho$-systems form a category in a very natural way. Of itself, it does not add
anything new to the construction, it just provides a conceptualisation which
will be useful at least once in Section~\ref{simpli}, but we
believe it may also prove useful in developing the theory further.
Throughout this section $\mathsf{Cat}$ will stand for the category of all
categories (with functors as arrows). For any
category $\mathsf{C}$, we will write $\mathrm{obj}(\mathsf{C})$ for the class of
objects of $\mathsf{C}$.  

\begin{definition}\label{t-over-S}
Let $\mathcal{S} = (\mathbf{I},\blambda,\brho)$ and
$\mathcal{S}'=(\mathbf{I}',\blambda',\brho')$  be
$\lambda\rho$-systems over a semigroup $\mathbf S$. By a
\emph{slice transformation} $t$ from $\mathcal{S}$ to $\mathcal{S}'$
we mean a system of maps 
$t = (t[a]\colon I[a]\longrightarrow I'[a])_{a\in S}$
satisfying $\lambda'[ab]\circ t[ab] = t[a]\circ \lambda[a,b]$ and
$\rho'[ab]\circ t[ab]= t[b]\circ \rho[a,b]$ for all $a,b\in S$, i.e., such that
the diagrams below commute. 
$$
\begin{tikzpicture}[>=stealth,auto]
\node (tab) at (0,0) {$I[ab]$};
\node (ab) at (3,0) {$I'[ab]$};
\node (ta) at (0,-2) {$I[a]$};
\node (a) at (3,-2) {$I'[a]$};
\draw[->] (tab) to node {$t[ab]$} (ab);
\draw[->] (tab) to node[swap] {$\lambda[a,b]$} (ta);
\draw[->] (ab) to node[swap] {$\lambda'[a,b]$} (a);
\draw[->] (ta) to node {$t[a]$} (a);
\end{tikzpicture} 
\qquad
\begin{tikzpicture}[>=stealth,auto]
\node (tab) at (0,0) {$I[ab]$};
\node (ab) at (3,0) {$I'[ab]$};
\node (tb) at (0,-2) {$I[b]$};
\node (b) at (3,-2) {$I'[b]$};
\draw[->] (tab) to node {$t[ab]$} (ab);
\draw[->] (tab) to node[swap] {$\rho[a,b]$} (tb);
\draw[->] (ab) to node[swap] {$\rho'[a,b]$} (b);
\draw[->] (tb) to node {$t[b]$} (b);
\end{tikzpicture} 
$$
\end{definition}

\begin{definition}\label{cat-lr-over-S}
Let $\mathbf{S}$ be a semigroup. We define $\bm{\lambda\rho}(\mathbf{S})$ to be  
the category whose objects are $\lambda\rho$-systems over a semigroup
$\mathbf{S}$, and whose arrows are slice transformations.
\end{definition}

It is clear that $\bm{\lambda\rho}(\mathbf{S})$
is a category: composition of slice transformations is a
slice transformation and the identity arrow is a system of identity maps.
In fact, it resembles a slice category---hence the terminology---but we
will not dwell on that. Having defined the category of $\lambda\rho$-systems
over a fixed semigroup, in the next step we will upgrade this definition to general 
$\lambda\rho$-systems over arbitrary semigroups. We will do it by means of
Grothendieck construction, whose one version we will now recall.

\begin{definition}[Grothendieck construction]
Let\/ $\mathsf{C}$ be an arbitrary category, and let
$F\colon \mathsf{C}^{op}\to\mathsf{Cat}$
be a functor. Then, $\mathsf{\Gamma}(F)$ is the category defined as follows.
\begin{enumerate}
\item Objects of\/ $\mathsf{\Gamma}(F)$ are pairs of $(A,X)$ such that
$A\in \mathrm{obj}(\mathsf{C})$ and $B\in \mathrm{obj}(F(A))$. 
\item Arrows between objects $(A_1,X_1),(A_2,X_2)\in
\mathrm{obj}(\mathsf{\Gamma}(F))$ are pairs
  $(f,g)$ such that $f\colon A_2\to A_1$ is an arrow in the category
  $\mathsf{C}$ and $g\colon F(f)(X_1)\to X_2$. 
\item Having objects and arrows in $\mathsf{\Gamma}(F)$, given below:
$$
(A_1,X_1)\overset{(f_1,g_1)}\longrightarrow(A_2,X_2)
\overset{(f_2,g_2)}\longrightarrow(A_3,X_3) 
$$ 
the composition of arrows is defined by:
$$
(f_1,g_1)\circ(f_2,g_2)=(f_2\circ f_1, g_1\circ F(f_1)(g_2)).
$$
\end{enumerate}
\end{definition}

To apply Grothendieck construction to $\lambda\rho$-systems, we first show the
existence of a suitable contravariant functor from (the opposite of the category
of) semigroups to categories.

\begin{lemma}
Let\/ $\mathsf{Sg}$ be the category of semigroups
\textup{(}with homomorphisms\textup{)}.  
There exists a functor $\bm{\lambda\rho}\colon \mathsf{Sg}^{op}\to
\mathsf{Cat}$ such that $\mathbf S\mapsto \bm{\lambda\rho}(\mathbf S)$, and
for each semigroup homomorphism $f\colon \mathbf S_1\to\mathbf S_2$
we have a functor 
$$
\bm{\lambda\rho}(f)\colon \bm{\lambda\rho}(\mathbf S_2)\longrightarrow
\bm{\lambda\rho}(\mathbf S_1)
$$
such that 
\begin{enumerate}
\item If $\mathcal{S} = (\mathbf{I},\blambda,\brho)\in
\bm{\lambda\rho}(\mathbf S_2)$, then
$$
\bm{\lambda\rho}(f)(\mathcal{S}) =
(\bm{\lambda\rho}(f)\mathbf{I},\bm{\lambda\rho}(f)\blambda,
\bm{\lambda\rho}(f)\brho)
$$ 
where 
\begin{eqnarray*}
\bm{\lambda\rho}(f)\mathbf{I} &=& \bigl(I[f(x)]\bigr)_{x\in S_1},\\
\bm{\lambda\rho}(f)\blambda &=&
\bigl(\lambda[f(x),f(y)]\colon
I[f(xy)]\to I[f(x)]\bigr)_{(x,y)\in S_1\times S_1},\\
\bm{\lambda\rho}(f)\brho &=&
\bigl(\rho[f(x),f(y)]\colon
I[f(xy)]\to I[f(y)]\bigr)_{(x,y)\in S_1\times S_1}.
\end{eqnarray*}
\item For any $\lambda\rho$-systems $\mathcal{S} = (\mathbf{I},\blambda,\brho)$ and
$\mathcal{S}' = (\mathbf{I}',\blambda',\brho')$ over a semigroup
$\mathbf{S}_2$, and for any slice transformation
$t\colon \mathcal{S}\to\mathcal{S}'$, such that 
$$
t = \bigl(f[x]\colon I[x]\longrightarrow I'[x]\bigr)_{x\in S_2}
$$
we have a slice transformation
$\bm{\lambda\rho}(f)t\colon
\bm{\lambda\rho}(f)(\mathcal{S})\to \bm{\lambda\rho}(f)(\mathcal{S}')$
such that
$$
\bm{\lambda\rho}(f)t =
\bigl(t[f(x)]\colon I[f(x)]\longrightarrow I'[f(x)]\bigr)_{x\in S_1}.
$$
\end{enumerate}
\end{lemma}
\begin{proof}
The proof is a series of tedious but straightforward calculations, which we
omit. A crucial point is that since $\bm{\lambda\rho}(f)$ acts contravariantly,
$\bm{\lambda\rho}(f)\mathbf{I}$, 
$\bm{\lambda\rho}(f)\blambda$ and $\bm{\lambda\rho}(f)\brho$ are 
well defined. For the proofs that ($\alpha$), ($\beta$) and ($\gamma$) are
satisfied, and that $\bm{\lambda\rho}(f)$ behaves properly on slice transformations,
we only need the definitions, the fact that $f$ is a homomorphism, and a lot of
paper.
\end{proof}

Now we are ready to define the notion of a 
\emph{transformation} between general $\lambda\rho$-systems.
Our definition may look a little esoteric, but
it is an appropriate notion of a morphism for general
$\lambda\rho$-systems. Firstly, it is natural for an application of
Grothedieck construction, and secondly, as we will show in the next section, it
is functorial for $\lambda\rho$-products as well. Slice
transformations, defined previously, are just a particular case 
of transformations. 

\begin{definition}\label{general-t}
Let $\mathcal{S} = (\mathbf{S}, \mathbf{I}, \boldsymbol{\lambda},
\boldsymbol{\rho})$ and $\mathcal{S}' = (\mathbf{S}', \mathbf{I}', 
\boldsymbol{\lambda}', \boldsymbol{\rho}')$ be general $\lambda\rho$-systems.
Define $\mathbf{t}\colon \mathcal{S}'\to \mathcal{S}$ to be a pair
$(t,h)$ consisting  
of a homomorphism $h\colon \mathbf{S}\to \mathbf{S}'$ and 
a system of maps $t = \bigl(t[a]\colon I'[{h(a)}] \to I[a]\bigr)_{a\in S}$,
such that the diagrams below commute.
$$
\begin{tikzpicture}[>=stealth,auto]
\node (tab) at (0,0) {$I'[{h(a)h(b)}] = I'[{h(ab)}]$};
\node (ab) at (3,0) {$I[{ab}]$};
\node (ta) at (0,-2) {$I'[{h(a)}]$};
\node (a) at (3,-2) {$I[a]$};
\draw[->] (tab) to node {$t[ab]$} (ab);
\draw[->] (tab) to node[swap] {$\lambda'[{h(a),h(b)}]$} (ta);
\draw[->] (ab) to node[swap] {$\lambda[{a,b}]$} (a);
\draw[->] (ta) to node {$t[a]$} (a);
\end{tikzpicture} 
\qquad
\begin{tikzpicture}[>=stealth,auto]
\node (tab) at (0,0) {$I'[{h(a)h(b)}] = I'[{h(ab)}]$};
\node (ab) at (3,0) {$I[{ab}]$};
\node (tb) at (0,-2) {$I'[{h(a)}]$};
\node (b) at (3,-2) {$I[b]$};
\draw[->] (tab) to node {$t[ab]$} (ab);
\draw[->] (tab) to node[swap] {$\rho'[{h(a),h(b)}]$} (tb);
\draw[->] (ab) to node[swap] {$\rho[{a,b}]$} (b);
\draw[->] (tb) to node {$t[b]$} (b);
\end{tikzpicture} 
$$
Any such pair $\mathbf{t} = (t,h)$ will be called a \emph{transformation}. 
\end{definition}

\begin{rem}
Let $\mathcal{S} = (\mathbf{S}, \mathbf{I}, \boldsymbol{\lambda},
\boldsymbol{\rho})$ and $\mathcal{S}' = (\mathbf{S}', \mathbf{I}', 
\boldsymbol{\lambda}', \boldsymbol{\rho}')$ be general $\lambda\rho$-systems. 
If $\mathbf{S} = \mathbf{S}'$, then
for any transformation $\mathbf{t}\colon \mathcal{S}'\to\mathcal{S}$ with
$\mathbf{t} = (t, id_S)$, we have that $t$ is a slice transformation.  
\end{rem}  

\begin{example}\label{subsystem}
Let $\mathcal{S} = (\mathbf{S}, \mathbf{I}, \boldsymbol{\lambda},
\boldsymbol{\rho})$ be a general $\lambda\rho$-system, and\/ let $\mathbf{T}$ be a
subsemigroup of\/ $\mathbf{S}$. Let $\mathbf{I}|_T$, $\blambda|_T$ and
$\brho|_T$ be the restrictions of\/ $\mathbf{I}$, $\blambda$ and
$\brho$ to $T$. Then
$\mathcal{T} = (\mathbf{T}, \mathbf{I}|_T,\blambda|_T,\brho|_T)$
is a $\lambda\rho$-system over $\mathbf{T}$. Moreover,
$\mathbf{t}\colon \mathcal{S}\to\mathcal{T}$, defined by
taking $h\colon\mathbf{T}\to\mathbf{S}$ to be the identity embedding,
together with the system $\bigl(t[a]\colon I[h(a)] \to I[a]\bigr)_{a\in T}$,
where $h(a) = a$ and $t[a] = {id}_{I[a]}$ is obviously a transformation.
\end{example}  

If $\mathcal{S}$ and $\mathcal{T}$ are $\lambda\rho$-systems related as in
Example~\ref{subsystem}, we will call $\mathcal{T}$ a \emph{subsystem} of
$\mathcal{S}$. We will sometimes write $\mathcal{S}|_T$ for a subsystem
of $\mathcal{S}$ over a semigroup $\mathbf{T}\leq\mathbf{S}$.

\begin{rem}\label{cat-general-lr}
Applying Grothendieck construction with
$\mathsf{C} = \mathsf{Sg}$ and
$F = \bm{\lambda\rho}$, we obtain
a category $\mathsf{\Gamma}(\bm{\lambda\rho})$ of general
$\lambda\rho$-systems with transformations as arrows.
\end{rem}

\section{Simplifications}\label{simpli}

If the semigroup $\mathbf{S}$ is in fact a monoid, any $\lambda\rho$-system
constructed over $\mathbf{S}$ will contain a set $I[1]$, and maps 
$\lambda[a,1]$, $\lambda[1,a]$, $\rho[a,1]$, $\rho[1,a]$ for any $a\in S$.
It is immediate from the defining equations
($\alpha$), ($\beta$) and ($\gamma$) that 
that the maps $\rho[1,a]$ and $\lambda[a,1]$ are
commuting retractions, that is, they satisfy
\begin{itemize}  
\item $\lambda[a,1]\circ\lambda[a,1] = \lambda[a,1]$
\item $\rho[1,a]\circ\rho[1,a] = \rho[1,a]$
\item $\lambda[a,1]\circ\rho[1,a] = \rho[1,a]\circ\lambda[a,1]$
\end{itemize}
for each $a\in S$. In fact, for monoids it is 
reasonable to require something stronger, but instead of stating
it for this particular case, we will define a general preservation
requirement, whose special case will apply to monoids. 

\begin{definition}\label{P-preserving}
Let $P$ be a property of semigroups, and let
$\mathcal{S} = (\mathbf{S},\mathbf{I},\blambda,\brho)$ be a $\lambda\rho$-system.
We will say that $\mathcal{S}$ \emph{preserves} $P$
\textup{(}or, \emph{is $P$ preserving}\textup{)}, if 
for every $\mathbf{H}$, whenever $\mathbf{H}$ satisfies $P$,
so does $\mathbf{H}^{[\mathcal{S}]}$. 
\end{definition}  

Said concisely, $\mathcal{S}$ is $P$ preserving, if
$\forall\mathbf{H}\colon P(\mathbf{H})\Rightarrow P(\mathbf{H}^{[\mathcal{S}]})$.
If $P$ is the property of having a unit, then 
$\mathcal{S}$ is $P$ preserving (\emph{unit-preserving})
if and only if $\mathbf{H}^{[\mathcal{S}]}$ is a monoid, for every monoid $\mathbf{H}$.

\begin{theorem}\label{main-monoid}
Let $\mathcal{S} = (\mathbf{S},\mathbf{I},\blambda,\brho)$ be a
$\lambda\rho$-system. The following are equivalent:
\begin{enumerate}
\item $\mathcal{S}$ is unit-preserving,
\item $\mathbf{S}$ is a monoid and the maps $\lambda[a,1]$ and $\rho[1,a]$ are
identity maps on $I[a]$, for each $a\in S$.
\end{enumerate}
\end{theorem}

\begin{proof}
Assume $\mathcal{S}$ is unit-preserving. Then, in particular,
$\mathbf{1}^{[\mathcal{S}]}$ is a monoid, so since
$\mathbf{1}^{[\mathcal{S}]}\cong \mathbf{S}$ (see Example~\ref{skel}), we get
that $\mathbf{S}$ is a monoid. Next, consider 
$\mathbf{H}^{[\mathcal{S}]}$ for an arbitrary nontrivial monoid $\mathbf{H}$.
By assumption, $\mathbf{H}^{[\mathcal{S}]}$ is a monoid, so let
$(y,b)$ be its unit element. In particular, $(y,b)\star(x,1) = (x,1)$
for any $x\in H^{I[1]}$, which implies $b\cdot 1 = 1$, so
$b = 1$. Next, taking $(1,a)$ for any $a\in S$ (where $1$ is the map from $I[a]$
to $H$ identically equal to $1$), we have
$(y,1)\star(1,a) = (1,a)$, from which we get
\begin{align*}
(1,a) &= \bigl((y\circ\lambda[1,a])\cdot(1\circ\rho[1,a])\ ,a\bigr)\\
      &= \bigl((y\circ\lambda[1,a])\cdot 1,\ a\bigr)\\
      &= (y\circ\lambda[1,a],\ a).
\end{align*}
This implies that $y\circ\lambda[1,a]$ is also identically $1$.
Thus, finally, we obtain
\begin{align*}
(x,a) &= (y,1)\star(1,a)\\
      &= \bigl((y\circ\lambda[1,a])\cdot(x\circ\rho[1,a])\ ,a\bigr)\\
      &= \bigl(1\cdot(x\circ\rho[1,a]),\ a\bigr)\\
      &= (x\circ\rho[1,a],\ a),
\end{align*}
and therefore $x = x\circ\rho[1,a]$. Since this holds for an arbitrary $x$,
$\rho[1,a] = {id}_{I[a]}$ as required. By symmetry, the same holds for
$\lambda[a,1]$, finishing the proof of (1) $\Rightarrow$ (2). 
  
For the converse, let $\mathbf{H}$ be any monoid.
Since $\mathbf{S}$ is a monoid, the set $I[1]$ exists; since $\mathbf{H}$ is a
monoid, the constant function $1$ belongs to $H^{I[1]}$. 
Then, we have 
\begin{align*}
(1,1)\star (x,1) &= \bigl((1\circ\lambda[1,a])\cdot (x\circ\rho[1,a]),\ 1\cdot 1\bigr)\\
&= (1\cdot (x\circ {id}_{I[a]}),\ 1\bigr)\\
&= (x, 1)
\end{align*}
showing that $(1,1)\in H^{I[1]}$ is a left unit. A completely symmetric argument
shows that it is a right unit as well. 
\end{proof}

If a $\lambda\rho$-system satisfies conditions of Theorem~\ref{main-monoid}(2),
we will call it \emph{unital}. This piece of terminology is, strictly speaking,
redundant, but we find it conceptually useful as a name for an intrinsic
characterisation of being unit-preserving.
Note that the $\lambda\rho$-system of Example~\ref{lzero} is not unital, but
the one of Example~\ref{flip-flop} is.

We will now make a series of observations that will eventually lead to
a simplification of $\lambda\rho$-systems to a singly indexed 
version, at a cost of some relatively mild assumptions.

Before we approach it, we
state two lemmas, which are rather obvious but they make
the transition between $\lambda\rho$-systems over semigroups and monoids
smooth. As usual, we will write $\mathbf{S}^1$ for the semigroup $\mathbf{S}$
with the unit element $1$ adjoined. We adopt the convention that
$1$ is always a new element, even if $\mathbf{S}$ already has a unit. 

\begin{definition}\label{add-unit} 
Let $\mathcal{S}$ be a $\lambda\rho$-system over a semigroup $\mathbf{S}$.
Let $1$ be an element not in $S$. Define $\mathcal{S}^1$ to be the following
system of maps
$$ 
\bigl(\langle\lambda[a,b],\rho[a,b]\rangle\colon 
I[ab]\to I[a]\times I[b]\bigr)_{(a,b)\in S^1\times S^1}
$$
where $I[1] = \{1\}$, the maps $\lambda[1,a]$, $\rho[a,1]$ are 
constant maps from $I[a]$ to $I[1]$, and the maps
$\lambda[a,1]$, $\rho[1,a]$ are both equal to ${id}_{I[a]}$.
\end{definition}

The $\lambda\rho$-system $\mathcal{S}^1$ defined above will be called
the \emph{unital extension} of $\mathcal{S}$. 
The terminology is justified by the next two lemmas, whose proofs are immediate.
An illustration of their application is provided by Examples~\ref{lzero}
and~\ref{flip-flop} in Section~\ref{intro}.

\begin{lemma}\label{unit-added} 
Let $\mathcal{S}$ be a $\lambda\rho$-system over a semigroup $\mathbf{S}$. 
The system $\mathcal{S}^1$ defined above is a unital $\lambda\rho$-system over
$\mathbf{S}^1$.  Moreover, $\mathcal{S}$ is a subsystem of $\mathcal{S}^1$. 
\end{lemma}

\begin{lemma}\label{lr-prod-with-unit}
Let $\mathcal{S}$ and $\mathcal{S}^1$ be $\lambda\rho$-systems defined above.
Let $\mathbf{H}$ be a semigroup. Then, $\mathbf{H}^{[\mathcal{S}]}$ is a
subsemigroup of  $\mathbf{H}^{[\mathcal{S}^1]}$.
\end{lemma}

For any unital $\lambda\rho$-system, in particular for $\mathcal{S}^1$, we can
simplify the definitions of 
the systems of maps $\lambda$ and $\rho$ by removing the double indexing. First,
defining $\lambda[a] = \lambda[1,a]$ and $\rho[a] = \rho[a,1]$, we split
the diagram in Figure~\ref{l-r-system} into three parts, substituting 
$1$ for $a$ in the left part, for $b$ in the middle part, and for $c$ in the
right part, as in Figure~\ref{split-l-r-system}. One should think about these
diagrams as three separate copies of the diagram from
Figure~\ref{l-r-system}, in which certain inessential fragments were 
suppressed. For example, in the left part of the diagram the maps 
$\rho[1,b]\colon I[b]\to I[b]$ and $\rho[1,bc]\colon I[bc]\to I[bc]$ are 
omitted, and so is the map $\rho[b,c]\colon I[bc]\to I[c]$. 
\begin{figure}
\begin{tikzpicture}[>=stealth,auto]
\node (I_abc_l) at (-1,0) {$I[bc]$};
\node (I_a) at (-5,0) {$I[1]$};
\node (I_ab_l) at (-3,-2) {$I[b]$};
\node (I_abc) at (0,-1) {$I[ac]$};
\node (I_ab) at (-2,-3) {$I[a]$};
\node (I_bc) at (2,-3) {$I[c]$};
\node (I_b) at (0,-5) {$I[1]$};
\node (I_abc_r) at (1,0) {$I[ab]$};
\node (I_c) at (5,0) {$I[1]$};
\node (I_bc_r) at (3,-2) {$I[b]$};
\draw[->] (I_abc_l) to node[swap] {$\lambda[bc]$} (I_a);
\draw[->] (I_abc_r) to node {$\rho[ab]$} (I_c);
\draw[->] (I_abc) to node[swap] {$\lambda[a,c]$} (I_ab);
\draw[->] (I_abc) to node {$\rho[a,c]$} (I_bc);
\draw[->] (I_abc_l) to node[swap] {$\lambda[b,c]$} (I_ab_l);
\draw[->] (I_abc_r) to node {$\rho[a,b]$} (I_bc_r);
\draw[->] (I_ab_l) to node {$\lambda[b]$} (I_a);
\draw[->] (I_ab) to node[swap] {$\rho[a]$} (I_b);
\draw[->] (I_bc_r) to node[swap] {$\rho[b]$} (I_c);
\draw[->] (I_bc) to node {$\lambda[c]$} (I_b);
\end{tikzpicture}
\caption{A $\lambda\rho$-system, split into three parts}
\label{split-l-r-system}
\end{figure}
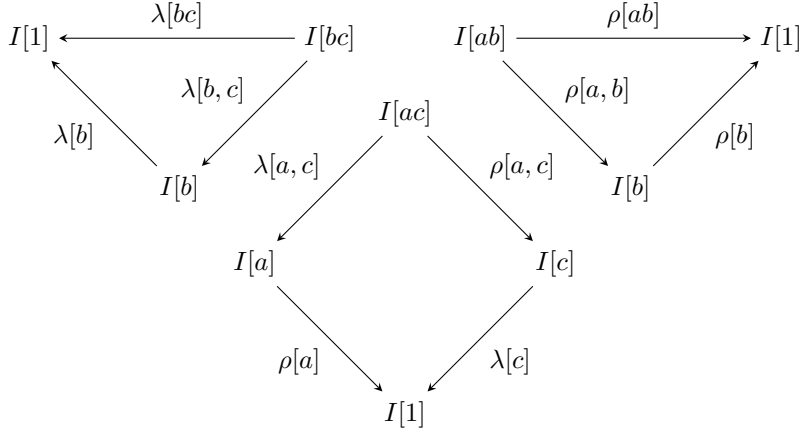

Next, renaming $a$, $b$ and $c$ as necessary, and putting the diagrams from 
Figure~\ref{split-l-r-system} together, 
we obtain the diagram of Figure~\ref{pre-l-r-system}, where the dashed lines
denote existence requirements, as usual, and where  
we write $I$ for $I[1]$.
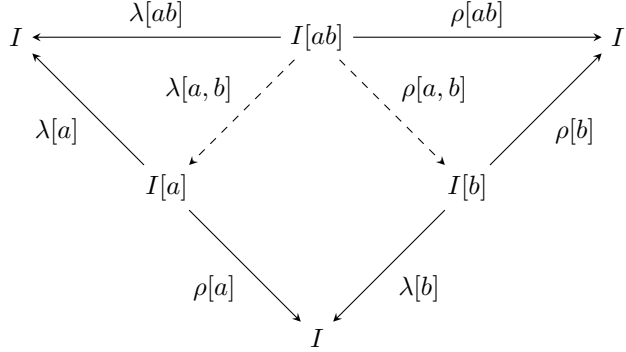
\begin{figure}
\begin{tikzpicture}[>=stealth,auto]
\node (I_abc) at (0,0) {$I[ab]$};
\node (I_a) at (-4,0) {$I$};
\node (I_c) at (4,0) {$I$};
\node (I_ab) at (-2,-2) {$I[a]$};
\node (I_bc) at (2,-2) {$I[b]$};
\node (I_b) at (0,-4) {$I$};
\draw[->] (I_abc) to node[swap] {$\lambda[ab]$} (I_a);
\draw[->] (I_abc) to node {$\rho[ab]$} (I_c);
\draw[dashed,->] (I_abc) to node[swap] {$\lambda[a,b]$} (I_ab);
\draw[dashed,->] (I_abc) to node {$\rho[a,b]$} (I_bc);
\draw[->] (I_ab) to node {$\lambda[a]$} (I_a);
\draw[->] (I_ab) to node[swap] {$\rho[a]$} (I_b);
\draw[->] (I_bc) to node[swap] {$\rho[b]$} (I_c);
\draw[->] (I_bc) to node {$\lambda[b]$} (I_b);
\end{tikzpicture}
\caption{A pre-$\lambda\rho$-system}
\label{pre-l-r-system}
\end{figure}

Now, let $\bigl(\langle\lambda[a],\rho[a]\rangle\colon I[a]\to I\bigr)_{a\in S^1}$
be a system of maps
and sets (with $I = I[1]$) such that for every $(a,b)\in S^1\times S^1$ there
exist maps 
$\lambda[a,b]$ and $\rho[a,b]$ making the required diagrams commute, that is,
satisfying the obvious counterparts of the
identities ($\alpha$), ($\beta$), and ($\gamma$), below. 
\begin{enumerate}
\item[($\alpha'$)] $\lambda[a]\circ\lambda[a,b] = \lambda[ab]$
\item[($\beta'$)] $\rho[b]\circ\rho[a,b] = \rho[ab]$
\item[($\gamma'$)] $\rho[a]\circ\lambda[a,b] = \lambda[b]\circ\rho[a,b]$
\end{enumerate}
We will refer to any such system by a rather unimaginative name of 
\emph{pre-$\lambda\rho$-system}. The rationale for the prefix `pre' will be
given shortly, but before that, let us make a few more observations. 
Given a pre-$\lambda\rho$-system, we can always define
$P[ab] = \{(x,y)\in I[a]\times I[b]\colon \rho[a](x) = \lambda[b](y)\}$, 
and obtain a pullback diagram 
$$
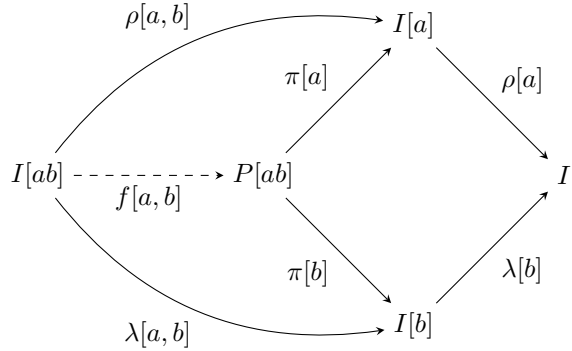
\begin{tikzpicture}[>=stealth,auto]
\node (P_ab) at (0,0) {$P[ab]$};
\node (I_ab) at (-3,0) {$I[ab]$};  
\node (I_a) at (2,2) {$I[a]$};
\node (I_b) at (2,-2) {$I[b]$};
\node (I) at (4,0) {$I$};
\draw[->] (P_ab) to node {$\pi[a]$} (I_a);
\draw[->] (P_ab) to node[swap] {$\pi[b]$} (I_b);
\draw[->] (I_b) to node[swap] {$\lambda[b]$} (I);
\draw[->] (I_a) to node {$\rho[a]$} (I);
\draw[bend left,->] (I_ab) to node {$\rho[a,b]$} (I_a);
\draw[bend right,->] (I_ab) to node[swap] {$\lambda[a,b]$} (I_b);
\draw[->,dashed] (I_ab) to node[swap] {$f[{a,b}]$} (P_ab);
\end{tikzpicture} 
$$
where $\pi[a]$ and $\pi[b]$ are projections.
Thus, from any pre-$\lambda\rho$-system we can obtain a `finest'
one  by systematically replacing $I[ab]$ by $P[ab]$, and factoring the
maps $\lambda$ and $\rho$ through. As this is not an
essential issue, we will not go into details.

\begin{lemma}\label{way-down}
Let $\mathbf{M}$ be a monoid, and let  
$$
\mathcal{S} = \bigl(\langle\lambda[a,b],\rho[a,b]\rangle\colon 
I[ab]\to I[a]\times I[b]\bigr)_{(a,b)\in M^2}
$$
be a $\lambda\rho$-system. Then, 
$$
\mathcal{P} = (\lambda[a],\rho[a]\colon I[a]\to I)_{a\in M}
$$
where $\lambda[a] = \lambda[1,a]$ and $\rho[a] = \rho[a,1]$
is a pre-$\lambda\rho$-system.
\end{lemma}

\begin{proof}
Obvious.
\end{proof}

Going the other way is not completely trivial, but it does work under certain
natural conditions. Let  
$\mathcal{P} =
\bigl(\langle\lambda[a],\rho[a]\rangle\colon I[a]\to I\bigr)_{a\in M}$
be a pre-$\lambda\rho$-system over some monoid $\mathbf{M}$. We say that 
$\mathcal{P}$ \emph{has natural solutions}, if the maps 
$\lambda[a,b]$, $\rho[a,b]$ satisfying the equations 
($\alpha'$), ($\beta'$), and ($\gamma'$), also satisfy four  
cancellativity properties, namely
\begin{enumerate}
\item[($\delta_1$)] $\lambda[a]\circ\lambda[a,b]\circ\lambda[ab,c] = 
\lambda[a]\circ\lambda[a,bc] \Rightarrow 
\lambda[a,b]\circ\lambda[ab,c] = \lambda[a,bc]$,
\item[($\delta_2$)] $\rho[b]\circ \rho[b,c]\circ\rho[a,bc] = 
\rho[b]\circ\rho[ab,c] \Rightarrow
\rho[b,c]\circ\rho[a,bc] = \rho[ab,c]$,
\item[($\delta_3$)] $\rho[b]\circ \rho[a,b]\circ\lambda[ab,c] =
\rho[b]\circ\lambda[b,c]\circ\rho[a,bc] \Rightarrow 
\rho[a,b]\circ\lambda[ab,c] = \lambda[b,c]\circ\rho[a,bc]$,
\item[($\delta_4$)] 
$\lambda[b]\circ\rho[a,b]\circ\lambda[ab,c] =
\lambda[b]\circ\lambda[b,c]\circ\rho[a,bc] \Rightarrow
\rho[a,b]\circ\lambda[ab,c] = \lambda[b,c]\circ\rho[a,bc]$.
\end{enumerate} 
These conditions may look somewhat esoteric, but they are just  
`prefixed' versions of ($\alpha$), ($\beta$) and ($\gamma$). Observe that 
the left-hand sides of ($\delta_i$) for $i=1,2,3,4$ hold in any 
$\lambda\rho$-system (with $\lambda[a] = \lambda[1,a]$ and 
$\rho[b] = \rho[b,1]$). Diagrammatically, they are presented in 
Figure~\ref{ns-pre-l-r-system}, where all possible commutations are postulated. 
We are not sure what the diagram resembles more: a parachute or a jellyfish. But
we promise we will not have any more complicated diagrams in the article. 

Importantly, the conditions above hold in three rather important cases: 
(a) when $\lambda[a]$ and $\rho[a]$ are injective maps, (b) when $\lambda[a]$
and $\rho[a]$ 
are actions of a semigroup on itself by left and right multiplication, and (c) 
in the construction over a free monoid, which will be given in
Subsection~\ref{free-constr}.  
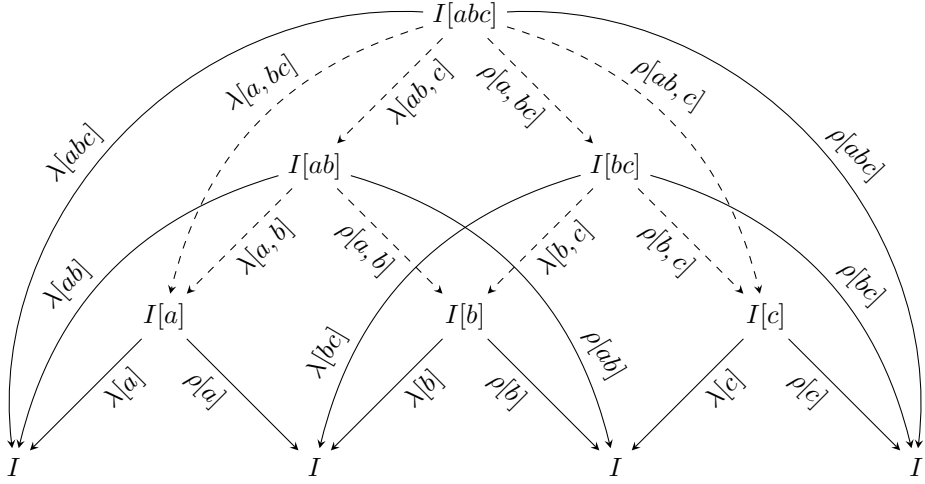
\begin{figure}
\begin{tikzpicture}[>=stealth,auto]
\node (I_abc) at (0,0) {$I[abc]$};
\node (I_a) at (-4,-4) {$I[a]$};
\node (I_c) at (4,-4) {$I[c]$};
\node (I_ab) at (-2,-2) {$I[ab]$};
\node (I_bc) at (2,-2) {$I[bc]$};
\node (I_b) at (0,-4) {$I[b]$};
\node (I_1) at (-6,-6) {$I$};
\node (I_2) at (-2,-6) {$I$};
\node (I_3) at (2,-6) {$I$};
\node (I_4) at (6,-6) {$I$};
\draw[->,bend right=50] (I_abc)  to node[swap,sloped] {$\lambda[abc]$} (I_1);
\draw[->,bend left=50] (I_abc) to node[sloped] {$\rho[abc]$} (I_4);
\draw[->,bend right] (I_ab) to node[swap,sloped] {$\lambda[ab]$} (I_1);
\draw[->,bend right] (I_bc) to node[pos=0.7,swap,sloped] {$\lambda[bc]$} (I_2);
\draw[->,bend left] (I_ab) to node[pos=0.7,sloped] {$\rho[ab]$} (I_3);
\draw[->,bend left] (I_bc) to node[sloped] {$\rho[bc]$} (I_4);
\draw[->] (I_a) to node[sloped] {$\lambda[a]$} (I_1);  
\draw[->] (I_a) to node[swap,sloped] {$\rho[a]$} (I_2);  
\draw[->] (I_b) to node[sloped] {$\lambda[b]$} (I_2);  
\draw[->] (I_b) to node[swap,sloped] {$\rho[b]$} (I_3);  
\draw[->] (I_c) to node[sloped] {$\lambda[c]$} (I_3);  
\draw[->] (I_c) to node[swap,sloped] {$\rho[c]$} (I_4);  
\draw[dashed,->,bend right] (I_abc) to node[pos=0.3,swap,sloped] {$\lambda[a,bc]$} (I_a);
\draw[dashed,->,bend left] (I_abc) to node[pos=0.3,sloped] {$\rho[ab,c]$} (I_c);
\draw[dashed,->] (I_abc) to node[pos=0.7,sloped] {$\lambda[ab,c]$} (I_ab);
\draw[dashed,->] (I_abc) to node[swap,pos=0.7,sloped] {$\rho[a,bc]$} (I_bc);
\draw[dashed,->] (I_ab) to node[pos=0.7,sloped] {$\lambda[a,b]$} (I_a);
\draw[dashed,->] (I_ab) to node[pos=0.7,swap,sloped] {$\rho[a,b]$} (I_b);
\draw[dashed,->] (I_bc) to node[pos=0.7,swap,sloped] {$\rho[b,c]$} (I_c);
\draw[dashed,->] (I_bc) to node[pos=0.7,sloped] {$\lambda[b,c]$} (I_b);
\end{tikzpicture}
\caption{A pre-$\lambda\rho$-system with natural solutions}
\label{ns-pre-l-r-system}
\end{figure}

\begin{lemma}\label{way-up}
Let $\mathbf{M}$ be a monoid, and let  
$$
\mathcal{P} = (\lambda[a],\rho[a]\colon I[a]\to I)_{a\in M}
$$
be a pre-$\lambda\rho$-system. If $\mathcal{P}$ has natural solutions, then
$$
\mathcal{S} = \bigl(\langle\lambda[a,b],\rho[a,b]\rangle\colon 
I[ab]\to I[a]\times I[b]\bigr)_{(a,b)\in M^2}
$$
where $\lambda[a,b]$ and $\rho[a,b]$ are some solutions to 
the equations \textup{(}$\alpha'$\textup{)}, \textup{(}$\beta'$\textup{)}, and
\textup{(}$\gamma'$\textup{)}, is a 
unital $\lambda\rho$-system.
\end{lemma}

\begin{proof}
To show that ($\alpha$) holds, first calculate, using ($\alpha'$)
\begin{align*}
\lambda[a]\circ\lambda[a,b]\circ\lambda[ab,c] &= 
\lambda[ab]\circ\lambda[ab,c] \\
&= \lambda[abc]\\
&= \lambda[a]\circ\lambda[a,bc]
\end{align*}
and then use ($\delta_1$) to cancel $\lambda[a]$ and obtain
$\lambda[a,b]\circ\lambda[ab,c] = \lambda[a,bc]$
as desired. By an analogous argument, ($\beta$) holds.

Now that ($\alpha$) and ($\beta$) have been shown to hold, we calculate, using
($\gamma'$) and ($\beta$)
\begin{align*}
\rho[b]\circ\rho[a,b]\circ\lambda[ab,c] &= 
\rho[ab]\circ\lambda[ab,c] \\
&= \lambda[c]\circ\rho[ab,c]\\
&= \lambda[c]\circ\rho[b,c]\circ\rho[a,bc]\\
&= \rho[b]\circ\lambda[b,c]\circ\rho[a,bc]
\end{align*}
ant then use ($\delta_3$) to cancel $\rho[b]$ and obtain
$\rho[a,b]\circ\lambda[ab,c] = \lambda[b,c]\circ\rho[ab,c]$ as desired.
\end{proof}

Combining Lemma~\ref{lr-prod-with-unit} with Lemma~\ref{way-up} we obtain
that any $\lambda\rho$-product is uniquely
determined by a pre-$\lambda\rho$-system with natural solutions.  
Therefore, we can---and will---extend the notation $\mathbf{H}^{[\mathcal{S}]}$
to the situation where $\mathcal{S}$ is such a system. 

\subsection{A free construction}\label{free-constr}
We will now show that that 
pre-$\lambda\rho$-systems with natural solutions, and thus 
$\lambda\rho$-systems, exist in abundance.  Let $X^*$ be
the free monoid, freely generated by some set $X$.
For any $x\in X$ we let $I[x]$ be a set,
and let $\lambda[x]\colon I[x]\to I$ and 
$\rho[x]\colon I[x]\to I$ be arbitrary maps. 
Then, we put $I[\varepsilon] = I$, and for each nonempty word
$w = x_1x_2\cdots x_k\in X^*$,   
we define $I[{x_1x_2\cdots x_k}]$ to be the set of sequences
$(v_1,v_2,\dots,v_k)\in I[{x_1}] \times \dots \times I[{x_k}]$ such that
\begin{align*}
\rho[{x_1}](v_1) &= \lambda[{x_2}](v_2) \\
\rho[{x_2}](v_2) &= \lambda[{x_3}](v_3) \\
 & \vdots  \\
\rho[{x_{k-1}}](v_{k-1}) &= \lambda[{x_k}](v_k). 
\end{align*}
To continue the construction, another piece of notation will be handy. 
Let $w = x_1x_2\cdots x_k$ be a word over $X$, and let  
$u = x_mx_{m+1}\cdots x_j$ be a subword 
of $w$, such that $1\leq m\leq j\leq k$.  Given a sequence 
$s = (v_{1},v_{2},\dots,v_{k})\in I[w]$, 
we write $s|_u$ for the truncated sequence
$(v_{m},\dots,v_{j})$. 
It s clear that $s|_u$ belongs to $I[u]$.

Now, we proceed inductively. First, we put $\lambda[\varepsilon] =
\rho[\varepsilon] = id_I$. The maps
$\lambda[z]$ and $\rho[z]$ for any $z\in X$ have already been defined.  
Let $w = w_1w_2$, with $w_1$ and $w_2$ nonempty. 
Assume the maps $\lambda[{w_i}]\colon I[{w_i}]\to I$ and 
$\rho[{w_i}]\colon I[{w_i}]\to I$ have been defined for $i\in \{1,2\}$. Define
$\lambda[{w}]\colon I[{w}]\to I$ and $\rho[{w}]\colon I[{w}]\to I$ by putting
\begin{align*}
\lambda[w](s) &= \lambda[{w_1}](s|_{w_1})\\
\rho[w](s) &= \rho[{w_2}](s|_{w_2})
\end{align*} 
for each $s\in I[w]$.

\begin{lemma}\label{free-is-free}
Let\/  $X^*$ be the free monoid generated by $X$, and 
for each element $w\in X^*$ let $I[w]$ be defined as above. 
Further, let
$$
\mathcal{F} =
\bigl(\langle\lambda[w],\rho[w]\rangle\colon I[w]\to I\bigr)_{w\in X^*}
$$ 
be the system of maps defined as above. Then $\mathcal{F}$
is a pre-$\lambda\rho$-system with natural solutions.
\end{lemma}

\begin{proof}
For a word $w\in X^*$ we denote its length by $|w|$. We will first prove, by 
induction on $|w|$, that the system we have defined is a pre-$\lambda\rho$-system.
The base case is $|w| = 2$. Then $w_1$ and $w_2$ are
generators of $X^*$, say, $w_1 = a$ and $w_2 = b$. Take an arbitrary
element $(v, u)\in I[{ab}]$. Then, by definition, 
$\lambda[{ab}](v,u) = \lambda[a](v)$ and $\rho[{ab}](v,u) = \rho[b](u)$.
Thus, taking $\lambda[{a,b}]\colon I[{ab}]\to I[a]$ to be the first projection, and 
$\rho[{a,b}]\colon I[{ab}]\to I[b]$ to be the second projection, we have that the
identities $(\alpha')$, $(\beta')$, $(\gamma')$ 
are satisfied, by the conditions imposed on $\lambda[a]$ and $\rho[b]$. 

For the inductive step assume 
$\lambda[{w_1}]\colon I[{w_1}]\to I$ and $\rho[{w_2}]\colon I[{w_2}]\to I$ have been
defined, and take an arbitrary element $s\in I[{w}] = I[{w_1w_2}]$.
Without loss of generality, we can assume $w_1 = u_1a$ and 
$w_2 = bu_2$, where $a,b\in H$. Then, by construction of 
$I[w]$ we have that $\rho[{a}](s|_{a}) = \lambda[{b}](s|_{b})$.
Now, $\rho[{w_1}] = \rho[{u_1a}]$, so, by inductive hypothesis, we obtain 
$\rho[{w_1}](s|_{w_1}) = \rho[{a}](s|_{a})$. Similarly,
$\lambda[{w_2}] = \lambda[{bu_2}]$, so we get 
$\lambda[{w_2}](s|_{w_2}) = \lambda[{b}](s|_{b})$. Therefore,
$\rho[{w_1}](s|_{w_1}) = \lambda[{w_2}](s|_{w_2})$. Then, 
taking $\lambda[{w_1,w_2}]\colon I[{w_1w_2}]\to I[{w_1}]$ to be the projection
onto $I[{w_1}]$, and $\rho[{w_1,w_2}]\colon I[{w_1w_2}]\to I[{w_2}]$ to be the
the projection on $I[{w_2}]$, we can see that $(\alpha')$, $(\beta')$, 
$(\gamma')$ are satisfied.

It remains to show that $\mathcal{F}$ has
natural solutions. To show that we only need to provide maps 
$\lambda[{ab,c}]$ and $\rho[{a,bc}]$ which make all diagrams in  
Figure~\ref{ns-pre-l-r-system} commute. But these maps have been already given
by our construction. Namely, taking $w_1 = ab$ and $w_2 = c$ we have that
$\lambda[{ab,c}]$ is the first projection from 
$I[{abc}]\subseteq I[{ab}]\times I[c]$ to $I[{ab}]$. Similarly, 
$\rho[{a,bc}]$ is the second projection from 
$I[{abc}]\subseteq I[a]\times I[{bc}]$ to $I[{bc}]$. Formally, the argument should
be again cast in the form of induction of the length of the word $w = abc$, but 
we are afraid it would then produce clutter rather than provide explanation.
\end{proof}

The final part of this section will show that the free construction
described above is indeed universal. Recall from
Section~\ref{cats}, Definition~\ref{general-t}, the notion of a 
transformation between $\lambda\rho$-systems.

\begin{theorem}\label{trans}
Let $\mathcal{S} = (\mathbf{S},\blambda,\brho)$ and
$\mathcal{S}' = (\mathbf{S}',\blambda',\brho')$ be $\lambda\rho$-systems, 
and let $\mathbf{t} = (t,h)$ be a transformation, with
$\mathbf{t}\colon \mathcal{S}'\to \mathcal{S}$.
For an arbitrary semigroup $\mathbf{H}$, let 
$\mathbf{H}^\mathbf{t}\colon 
\mathbf{H}^{[\mathcal{S}]} \to \mathbf{H}^{[\mathcal{S}']}$ be the map defined 
by $\mathbf{H}^\mathbf{t}(x,a) = (x\circ t[a],\ h(a))$
for every $(x,a)\in\biguplus_{a\in S} H^{I[a]}$. 
Then, $\mathbf{H}^\mathbf{t}$ is a homomorphism. Moreover,
$\mathbf{H}^{-}$ is a contravariant functor from the category
$\Gamma(\blambda\brho)$ to the category $\mathsf{Sg}$ of semigroups.
\end{theorem}

\begin{proof}
It is clear that the map $\mathbf{H}^\mathbf{t}$ is well defined.
Let $(x,a), (y,b)\in \biguplus_{a\in S} H^{I[a]}$. Then, we have
\begin{align*}
\mathbf{H}^\mathbf{t}(x\star y,\ ab)
&= \bigl((x\star y)\circ t[ab],\ h(ab)\bigr)\\
&= \Bigl(\bigl((x\circ\lambda[{a,b}])(y\circ\rho[{a,b}])\bigr)\circ t[ab],\
    h(a)h(b)\Bigr)\\
&= \Bigl((x\circ\lambda[{a,b}]\circ t[ab])(y\circ\rho[{a,b}]\circ t[ab]),\
h(a)h(b)\Bigr) \\
&= \Bigl(\bigl(x\circ t[a]\circ \lambda'[{h(a),h(b)}]\bigr)
    \bigl(y\circ t[b]\circ\rho'[{h(a),h(b)}]\bigr),\ h(a)h(b)\Bigr)\\
&= \bigl(x\circ t[a],\ h(a)\bigr)\star\bigl(y\circ t[b],\ h(b)\bigr)\\
&= \mathbf{H}^\mathbf{t}(x,a)\star \mathbf{H}^\mathbf{t}(y,b). 
\end{align*}
This proves that $\mathbf{H}^\mathbf{t}$ is a homomorphism. The proof of the
moreover part is straightforward. 
\end{proof}

Consider $\lambda\rho$-system $\mathcal{S}$ over some semigroup 
$\mathbf{S}$. Taking $S$ as the set of free generators, we form 
the free monoid $S^*$. Then, $\mathbf{S}^1$ is a homomorphic image
(in fact, a retract) of $S^*$ via the map extending the identity map
on $S$. Let $\mathcal{S}^1$ be the $\lambda\rho$-system over
$\mathbf{S}^1$, extending $\mathcal{S}$, as in Definition~\ref{add-unit}.
Next, let $\mathcal{P}^1$ be the pre-$\lambda\rho$-system associated
with $\mathcal{S}^1$ as in Lemma~\ref{way-down}. As
$\mathcal{P}^1$ is just a restriction of $\mathcal{S}^1$, it has natural
solutions, so by Lemma~\ref{way-up} it induces a unital $\lambda\rho$-system.
We will denote that system by $\mathcal{F}(\mathcal{S}^1)$. 

Since every element of $S^*$ is a word $s_1s_2\dots s_n$ over $S$, but on
the other hand $s_1s_2\cdots s_n$ also represents a product of $s_1,\dots,s_n$ 
as an element of $S$, we need some notational device to 
distinguish the two. We will write $s_1s_2\dots s_n$ for the word, 
and $\otimes(s_1s_2\dots s_n)$ for the product. Thus, for example,
$I[{\otimes(s_1s_2\dots s_n)}]$ will be a set from the original system $\mathcal{S}$,
and $I[{s_1s_2\dots s_n}]$ will be a set from the system 
$\mathcal{F}(\mathcal{S}^1)$. Further, $I[\otimes\varepsilon] = I[1]$ is a
singleton set from $\mathcal{S}^1$, and, by construction of
$\mathcal{F}(\mathcal{S}^1)$, the same set as a
member of $\mathcal{F}(\mathcal{S}^1)$ should be denoted by
$I[\varepsilon]$. We will simply write $I$ in either case.
Also  recall that, by construction of 
$\mathcal{F}(\mathcal{S}^1)$, we have 
$I[{s_1s_2\dots s_n}]\subseteq I[{s_1}]\times I[{s_2}]\times\dots I[{s_n}]$.

\begin{definition}\label{free-t}
Let $\mathcal{S}$, $\mathcal{S}^1$ and $\mathcal{F}(\mathcal{S}^1)$ be as above.
We define a system  $\mathbf{t}$ of maps as follows. 
First, we let $t\colon S^*\to \mathbf{S}^1$ be the 
homomorphism extending the identity map on $S$,
and such that $t(\varepsilon) = 1$. 
Next, for any $s_1,s_2,\dots, s_n \in S$, we  
define the map
$$
t[s_1s_2\dots s_n]\colon I[\otimes(s_1s_2\dots s_n)] \longrightarrow
I[{s_1}]\times I[{s_2}]\times\dots\times I[{s_n}]
$$
for each $v\in I[\otimes(s_1s_2\cdots s_n)]$, by putting
$t[s_1s_2\dots s_n](v) = \langle v_1,\dots, v_n\rangle$, where
\begin{itemize}
\item $v_1 = \lambda[s_1, \otimes(s_2s_3\cdots s_n)](v)$,
\item $v_j = \rho[\otimes(s_1\cdots s_{j-1}), s_j]\circ 
\lambda[\otimes(s_1\cdots s_j), \otimes(s_{j+1}\cdots s_n)](v)$, 
\item $v_n = \rho[\otimes(s_1\cdots s_{n-1}), s_n](v)$.
\end{itemize}
Finally, we let $t[\varepsilon]\colon I\to I$
be the (unique) constant map.
\end{definition}  

\begin{lemma}\label{well-defd}
$\mathcal{S}$, $\mathcal{S}^1$ and $\mathcal{F}(\mathcal{S}^1)$ be as above.  
Then, the following hold:
\begin{enumerate}
\item For each $s\in S$, we have $t[s] = {id}_{I[s]}$.
\item For any $s_1,s_2,\dots, s_n \in S$, we have
\begin{align*}
v_j &= \rho[\otimes(s_1\cdots s_{j-1}), s_j]\circ 
\lambda[\otimes(s_1\cdots s_j), \otimes(s_{j+1}\cdots s_n)](v)\\
&= \lambda[s_j, \otimes(s_{j+1}\cdots s_n)]\circ 
\rho[\otimes(s_1\cdots s_{j-1}),\otimes(s_{j}\cdots s_n)](v)
\end{align*}
for each $j\in \{2,\cdots,n-1\}$.
\end{enumerate}
\end{lemma}

\begin{proof}
For any $s\in S$, we have $I[s] = I[\otimes s]$ by construction.
Without loss of generality, let $s = s_1$.  
Take a $v\in I[s_1]$. By definition we have
$t[s_1](v) = \lambda[s_1,\otimes\varepsilon](v) = \lambda[s_1,1](v) = v$
because $\mathcal{S}^1$ is the unital extension of $\mathcal{S}$. This proves (1).
Next, (2) follows easily from the fact that $\lambda$ and $\rho$ come from a 
$\lambda\rho$-system. 
\end{proof}

\begin{lemma}
Let $\mathcal{S}$, $\mathcal{S}^1$, $\mathcal{F}(\mathcal{S}^1)$
and $\mathbf{t}$ be as above.  
Then, $\mathbf{t}\colon\mathcal{S}^1\to \mathcal{F}(\mathcal{S}^1)$
is a transformation. 
\end{lemma}

\begin{proof}
We need to show: (i) that the range of each map $t[s_1s_2\cdots s_n]$ belongs 
to  $I[{s_1s_2\cdots s_n}]$, and (ii) that the appropriate diagrams commute. 
To show (i), calculate:
\begin{align*}
\rho[{s_1}](v_1) &= \rho[{s_1,1}](v_1) \\
&=  \rho[s_1,1]\circ\lambda[s_1,\otimes(s_2s_3\cdots s_n)](v) \\
&= \rho[s_1,1]\circ\lambda[s_1,s_2]\circ
   \lambda[\otimes(s_1s_2),\otimes(s_3\cdots s_n)](v)\\ 
&= \lambda[1,s_2]\circ\rho[s_1,s_2]\circ
   \lambda[\otimes(s_1s_2),\otimes(s_3\cdots s_n)](v)\\ 
&= \lambda[1,s_2](v_2)\\
&= \lambda[s_2](v_2)
\end{align*}
where the first and last equalities are respectively the definitions of 
$\rho[s_1]$ and $\lambda[s_2]$, the second and fifth are
the definitions of $v_1$ and $v_2$, and the
third and fourth follow from ($\alpha$) and ($\beta$).
Similarly, but cutting a few corners now, we calculate:
\begin{align*}
\rho[s_j](v_j) &=  \rho[s_j,1]\circ\lambda[s_j,\otimes(s_{j+1}\cdots s_n)](v) \\
&= \lambda[1,\otimes(s_{j+1}\cdots s_n)]\circ\rho[s_j,\otimes(s_{j+1}\cdots s_n]
\circ\rho[\otimes(s_1\cdots s_{j-1}),\otimes(s_j\cdots s_n)](v)\\ 
&= \lambda[1,\otimes(s_{j+1}\cdots s_n)]\circ
   \rho[\otimes(s_1\cdots s_j),\otimes(s_{j+1}\cdots s_n)](v)\\
&= \lambda[1,s_{j+1}]\circ\lambda[s_{j+1},\otimes(s_{j+2}\cdots s_n)]\circ
\rho[\otimes(s_1\cdots s_j),\otimes(s_{j+1}\cdots s_n)](v)\\
&= \lambda[1, s_{j+1}](v_{j+1})\\
&= \lambda[s_{j+1}](v_{j+1})
\end{align*}
and
\begin{align*}
\rho[s_{n-1}](v_{n-1}) &= \rho[s_{n-1},1]\circ\lambda[s_{n-1},s_n]
\circ \rho[\otimes(s_1\cdots s_{n-2}),\otimes(s_{n-1}s_n)](v)\\
&= \lambda[1,s_n]\circ\rho[s_{n-1},s_n]
\circ \rho[\otimes(s_1\cdots s_{n-2}),\otimes(s_{n-1}s_n)](v)\\
&= \lambda[1,s_n]\circ\rho[\otimes(s_1\cdots s_{n-1}),s_n](v)\\
&= \lambda[s_n](v_n).
\end{align*}

For (ii), let $v\in I[\otimes(s_1s_2\cdots s_n)]$ and 
let $u = \lambda[\otimes(s_1\cdots s_k),\otimes(s_{k+1}\cdots s_n)](v)$. With this,
commutativity of the relevant diagram amounts to the  
equality between $\langle v_1,\dots, v_k\rangle$ (the first 
$k$ coordinates of $\langle v_1,\dots, v_n\rangle$) and 
$\langle u_1,\dots, u_k\rangle$. To verify that these indeed hold, we calculate
\begin{align*}
v_1 &= \lambda[s_1,\otimes(s_2\cdots s_n)](v)\\
&= \lambda[s_1,\otimes(s_2\cdots s_k)]\circ
\lambda[\otimes(s_1\cdots s_k),\otimes(s_{k+1}\cdots s_n)](v)\\
&= \lambda[s_1,\otimes(s_2\cdots s_k)](u)\\
&= u_1
\end{align*}
then, for $j\in\{2,\dots,k-1\}$
\begin{align*}
v_j &= \rho[\otimes(s_1\cdots s_{j-1}),s_j]\circ 
\lambda[\otimes(s_1\cdots s_j),\otimes(s_{j+1}\cdots s_n)](v)\\
&= \rho[\otimes(s_1\cdots s_{j-1}),s_j]\circ 
\lambda[\otimes(s_1\cdots s_j),\otimes(s_1\cdots s_k)] \circ 
\lambda[\otimes(s_1\cdots s_k),\otimes(s_{k+1}\cdots s_n)](v)\\
&= \rho[\otimes(s_1\cdots s_{j-1}),s_j]\circ 
\lambda[\otimes(s_1\cdots s_j),\otimes(s_1\cdots s_n)](u)\\
&= u_j
\end{align*}
and finally
\begin{align*}
v_k &= \rho[\otimes(s_1\cdots s_{k-1}), s_k]
\circ \lambda[\otimes(s_1\cdots s_k),\otimes(s_{k+1}\cdots s_n)](v)\\
&= \rho[\otimes(s_1\cdots s_{k-1}),s_k](u)\\
&= u_k
\end{align*}
proving (ii). 
\end{proof}

\begin{theorem}\label{divide}
Let $\mathcal{S}$ be a $\lambda\rho$-system over a semigroup $\mathbf{S}$, and
let $\mathbf{H}$ be a semigroup. Let
$\mathbf{t}\colon \mathcal{S}^1\to \mathcal{F}(\mathcal{S}^1)$
be the transformation from Definition~\ref{free-t}. Then,
$$
\mathbf{H}^{\mathbf{t}}\colon \mathbf{H}^{[\mathcal{F}(\mathcal{S}^1)]}
\longrightarrow \mathbf{H}^{[\mathcal{S}^1]}
$$
defined as in Lemma~\ref{trans}, is a surjective homomorphism. 
\end{theorem}  

\begin{proof}
The map $\mathbf{H}^{\mathbf{t}}$ is a homomorphism by Lemma~\ref{free-t}.
Surjectivity follows from Lemma~\ref{well-defd}(1). 
\end{proof}

\begin{cor}
Let $\mathcal{S}$ be a $\lambda\rho$-system over a semigroup $\mathbf{S}$, and
let $\mathbf{H}$ be a semigroup. Then,
$\mathbf{H}^{[\mathcal{S}]}\in SH(\mathbf{H}^{[\mathcal{F}(\mathcal{S}^1)]})$.
Therefore, $\mathbf{H}^{[\mathcal{S}]}$ divides $\mathbf{H}^{[\mathcal{F}(\mathcal{S}^1)]}$.
\end{cor}

\begin{proof}
By Lemma~\ref{lr-prod-with-unit}, and the well known universal algebraic
fact that $SH\leq HS$.
\end{proof}

\begin{rem}
The free construction we presented could be made even `freer' by taking 
$I[s_1\dots s_n] = I[s_1]\times\dots\times I[s_n]$. But then
Lemma~\ref{well-defd}(1) is no longer true, and consequently we lose surjectivity
in Theorem~\ref{divide}. 
\end{rem}  

\section{Wreath products}\label{wreath}

We will now return to the promise made in the introduction and prove that
every wreath product can be realised as a $\lambda\rho$-product.
Let $(X,\mathbf G)$ consist of a set $X$ and a group $\mathbf{G}$
acting on $X$ on the left, so that
$(x\cdot a)\cdot b = x\cdot ab$, for every $x\in X$ and every $a,b\in G$.
For any such $(X,\mathbf G)$ and any group $\mathbf{H}$ recall that 
their wreath product is defined as 
$$
\mathbf H\wr (X,\mathbf G)=(H^X\times G, *)
$$ 
with multiplication 
$$
(a,x)*(b,y)= (a\cdot (b\circ(\bl \cdot x)),\ xy).
$$
It is easy to see, that any $(X,\mathbf G)$ can be viewed by a
$\lambda\rho$-system
$$
\mathcal{S}(X,\mathbf G) =
\bigl(\langle \lambda[a,b],\rho[a,b]\rangle\colon
I[ab]\to I[a]\times I[b]\bigr),
$$
where $I[s]=X$ for any $s\in S$, and 
\begin{enumerate}
\item $\lambda[a,b] = {id}_X$ for any $a,b\in S$,
\item $\rho[a,b] = \bl\cdot a$ for all $a,b\in S$. 
\end{enumerate}
as in Example~\ref{sgrp-act}. Then 
$\mathbf H^{[\mathcal S(X,\mathbf G)]}\cong \mathbf{H}\wr (X,\mathbf G)$.

\begin{theorem}\label{wr-prod}
Let $\mathcal{S} = (\mathbf{G},\mathbf{I},\blambda,\brho)$  be a
$\lambda\rho$-system. The, the following are equivalent:
\begin{enumerate}
\item $\mathcal{S}$ is group-preserving,
\item $\mathbf{G}$ is a group and $\mathcal{S}$ is unital,
\item $\mathcal{S}\cong \mathcal{S}(X,\mathbf{G})$
for some group $\mathbf{G}$ acting on some set $X$. 
\end{enumerate}
\end{theorem}
\begin{proof}
Recall from Definition~\ref{P-preserving} that group-preserving means
$\mathbf{H}^{[\mathcal{S}]}$ is a group for any group $\mathbf{H}$.  

\smallskip\noindent
(1) $\Rightarrow$ (2). Group-preserving
$\lambda\rho$-systems preserve units, so by Theorem~\ref{main-monoid}
$\mathcal{S}$ is unital. Since $\mathbf{1}$ is a (trivial) group and $\mathcal{S}$ is
group-preserving then $\mathbf{1}^{[\mathcal S]}\cong \mathbf{G}$, and so
$\mathbf{G}$ is a group.   

\smallskip\noindent
(2) $\Rightarrow$ (3).
Let $\mathcal{S} = (\mathbf{G},\mathbf{I},\blambda,\brho)$  be a
unital $\lambda\rho$-system, with $\mathbf{G}$ a group.
Since $\mathcal{S}$ is unital, we have
$$
{id}_{I[x]}=\lambda[x,e]=\lambda [x,yy^{-1}] =
\lambda [x,y]\circ \lambda[xy,y^{-1}]
$$
for all $x,y\in G$ ($e$ is the unit of $\mathbf{G}$, of course).
Consequently, $\lambda [x,y]$ is surjective
and $\lambda [xy,y^{-1}]$ is injective for all $x,y\in G$. However,
$\lambda[x,y]=\lambda [xyy^{-1},y]$ and thus $\lambda [x,y]$ is a
bijection. Analogously we can prove bijectivity of $\rho[x,y]$. 
 
\begin{claim}
Consider the pair $(I[e],\mathbf G)$, 
The operation $\cdot\colon I[e]\times
G\longrightarrow I[e]$ defined by  
$$
i\cdot x = (\rho[x,e]\circ\lambda[e,x]^{-1})(i)
$$
is a group action.
\end{claim}
\begin{proof}
Substituting equalities
\begin{eqnarray*}
\lambda[ex,y]&=&\lambda[e,x]^{-1}\circ \lambda[e,xy]\\
\rho[x,ye]&=&\rho[y,e]^{-1}\circ \rho[xy,e]
\end{eqnarray*}
into the equality
$$
\lambda[e,y]\circ \rho[x,e\cdot y]=\rho[x,e]\circ\lambda[xe,y]
$$
we obtain
$$
\lambda [e,y]\circ \rho[y,e]^{-1}\circ
\rho[xy,e]=\rho[x,e]\circ\lambda[e,x]^{-1}\circ\lambda[e,xy]
$$
and hence
$$
\rho[y,e]\circ\lambda[e,y]^{-1} \circ\rho[e,x]\circ  \lambda
[e,x]^{-1}=\rho[xy,e]\circ \lambda[e,xy]^{-1}.\eqno{(\ddag)}
$$
It is easy to see that the last equality implies $(i\cdot x)\cdot y= i\cdot
(x\cdot y)$ for all $i\in I[e]$ and $x,y\in G$. 
\end{proof}

We will show that the system of bijections
$\mathbf{t} = \bigl(\lambda[e,x]\colon I[x]\longrightarrow I[e]\bigr)_{x\in G}$ form a
transformation and thus an isomorphism of $\lambda\rho$-systems $\mathcal{S}$ and
$\mathcal{S}(I[e],\mathbf{G})$. We need to prove 
commutativity of the following diagrams: 
$$
\begin{tikzpicture}[>=stealth,auto]
\node (tab) at (0,0) {$S(xy)$};
\node (ab) at (3,0) {$S(e)$};
\node (ta) at (0,-2) {$S(x)$};
\node (a) at (3,-2) {$S(e)$};
\draw[->] (tab) to node {$\lambda[e,xy]$} (ab);
\draw[->] (tab) to node[swap] {$\lambda [x,y]$} (ta);
\draw[->] (ab) to node[swap] {$id_{S(e)}$} (a);
\draw[->] (ta) to node {$\lambda[e,x]$} (a);
\end{tikzpicture} 
\qquad
\begin{tikzpicture}[>=stealth,auto]
\node (tab) at (0,0) {$S(xy)$};
\node (ab) at (4,0) {$S(e)$};
\node (tb) at (0,-2) {$S(x)$};
\node (b) at (4,-2) {$S(e).$};
\draw[->] (tab) to node {$\lambda[e,xy]$} (ab);
\draw[->] (tab) to node[swap] {$\rho[x,y]$} (tb);
\draw[->] (ab) to node[swap] {$\rho[x,e]\circ\lambda[e,x]^{-1}$} (b);
\draw[->] (tb) to node {$\lambda[e,y]$} (b);
\end{tikzpicture} 
$$
Commutativity of the first diagram is clear. Applying
$\rho[x,y]=\rho[y,e]^{-1}\circ\rho[xy,e]$ to $(\ddag)$, we obtain 
$$\lambda[e,y]^{-1}\circ\rho[e,x]\circ\lambda[e,x]^{-1}=\rho[x,y]\circ\lambda [e,xy]^{-1}$$
which proves commutativity of the second diagram. 

\smallskip\noindent
(3) $\Rightarrow$ (1). Follows from the fact that wreath product of groups
is a group. 
\end{proof}

Theorem~\ref{wr-prod} shows that $\lambda\rho$-products of groups over groups coincide
with wreath products. We already saw that for semigroups the
notion of a $\lambda\rho$-products is more general. In particular,
the two-sided wreath product of semigroups (see, e.g.,~\cite{RT89})
can be accommodated. 

\begin{example}\label{wr-two-sided}
Let $\mathcal{S}(X,\ld,\rd,\mathbf{S})$ be the $\lambda\rho$-system of
Example~\ref{sgrp-act-two-sided}.  Then,
$\mathbf{H}^{\mathcal S(X,\backslash,/,\mathbf S)}$ is (isomorphic to) the two-sided wreath
product of\/ $\mathbf{H}$ and $\mathbf{S}$.  
\end{example}

Combining Krohn-Rhodes Theorem, Theorem~\ref{wr-prod}, and
Example~\ref{flip-flop}, we get our final result.

\begin{cor}
Every finite semigroup divides an iterated
$\lambda\rho$-product whose factors are finite simple groups and a two-element
semilattice.
\end{cor}

\section{Acknowledgement}

This project has received funding from the European Union’s Horizon 2020
research and innovation programme under the Marie Skłodowska-Curie grant
agreement No.~689176.

\end{document}